\documentclass[12pt]{article}
\usepackage{amsmath}
\usepackage{amsthm}
\usepackage{amsfonts}
\usepackage{amssymb}
\usepackage{graphicx}
\usepackage{xypic}
\usepackage{a4wide}
\usepackage{enumerate}
\usepackage{mathrsfs}

\newtheorem{theorem}{Theorem}

\newtheorem{corollary}[theorem]{Corollary}
\newtheorem{lemma}[theorem]{Lemma}
\theoremstyle{definition}

\newtheorem{remark}[theorem]{Remark}

\newcommand{\se}[1]{\operatorname{S}_{#1}}

\begin{document}

\makeatletter
\renewcommand{\thesection}{\@arabic\c@section}
\makeatother

\title{On the Bourbaki's fixed point theorem and the axiom of choice}
\author{Mohssin Zarouali-Darkaoui
\\Department of Mathematics, \\
University of Sultan Slimane, Morocco\\
E-mail address: \textsf{mohssin.zarouali@gmail.com}}

\maketitle

\abstract
In this note we generalize the Moroianu's fixed point theorem. We propose a very elegant common proof of the Bourbaki's fixed point theorem and our result. We apply our result to give a very elegant proof of the fact that, in the Zermelo-Fraenkel system, the axiom of choice is equivalent to each of the following statements: H. Kneser's Lemma, Zorn's Lemma, Zermelo's Lemma.



\section{Introduction}

In mathematics, the axiom of choice, abbreviated as ``AC'', is an axiom of set theory which affirms the possibility of constructing sets by repeating an infinite number of times an action of choice, even of not explicitly specified. It was formulated for the first time by Ernest Zermelo in 1904 (see \cite{Zermelo:1904}). This axiom can be accepted or rejected, according to the axiomatic theory of sets chosen.

We recall that the Bourbaki's fixed point theorem \cite{Bourbaki:1949} has many important application. The Moroianu's fixed point theorem \cite{Moroianu:1971} is a slight version of the Bourbaki's theorem. In this note, we give a version of Moroianu's fixed point theorem \cite{Moroianu:1971} without assuming the AC (see the first part of Theorem \ref{Bourbaki-Moroianu}). We give a very elegant common proof of our result and the Bourbaki's theorem. We apply our version of the Moroianu's theorem to give a very elegant proof that, in the Zermelo-Fraenkel system, the axiom of choice has several equivalent statements: H. Kneser's Lemma \cite{Kneser:1950}, Zorn's Lemma \cite{Zorn:1935}, Zermelo's Lemma \cite{Zermelo:1904}.

\section{Preliminaries}

Let $E$ be a preordered set. A subset $S$ of $E$ is called \emph{initial segment} (or simply \emph{segment}) of $E$ if $$(x\in S\text{ and }y<x) \Longrightarrow y\in S.$$
For instance, $E$, the empty set, and the intervals $]{\leftarrow},a]$, $\se{a}=\,]{\leftarrow},a[$ are segments. If $S'$ is a segment of $S$ which is a segment of $E$, therefore $S'$ is a segment of $E$. Every union and every intersection of a family of segments is also a segment. Analogously we define \emph{final segment}.

$E$ is said to be a \emph{well-ordered set} if every non-empty subset $A$ of $E$ has a least element. We recall that the segments of a well-ordered set $E$ are: $E$, $\se{a}$ with $a\in E$.

$E$ is said to be \emph{inductive} if every totally ordered subset of it has an upper bound in $E$.

\begin{lemma}\label{réunion_Ens_BO}
Let $E$ be a set, and $\mathcal{F}$ be the set of all subsets of $E$ that can be well-ordered ($\varnothing\in\mathcal{F}$). We define a preorder on $\mathcal{F}$ as follows: $(A,\leq_A)\preceq(B,\leq_B)$ if
\begin{enumerate}[a)]
  \item $A\subset B$;
  \item $\leq_B$ induces $\leq_A$ on $A$;
  \item $A$ is a segment of $B$.
\end{enumerate}
Let $(A_i)_{i\in I}$ be a family of elements of $\mathcal{F}$ totally ordered by inclusion. The well-orderings on the sets $A_i$ can be extended to a well-ordering on $A=\bigcup_{i\in I}A_i$ and $(A_i,\leq_{A_i})\preceq(A,\leq_A)$ for all $i$.
\end{lemma}

\begin{proof}
Straightforward.
\end{proof}

The following (\cite[Lemme 3, III.19]{Bourbaki:1970}) is essentially due to A. Tarski (see \cite[Theorem 3]{Tarski:1939}).

\begin{lemma}[Tarski--Bourbaki]\label{lem_ZZC}
Let $E$ be a set, $\mathcal{S}\subset\mathcal{P}(E)$, and $\varphi:\mathcal{S}\to E$ a map such that $\varphi(X)\notin X$ for all $X\in\mathcal{S}$. Therefore there is a unique subset $M$ of $E$ that can be well-ordered satisfying
\begin{enumerate}[1)]
  \item for all $x\in M$: $\se{x}\in\mathcal{S}$ and $\varphi(\se{x})=x$;
  \item $M\notin\mathcal{S}$.
\end{enumerate}
\end{lemma}

\begin{proof}
Let $\mathcal{F}'$ be the set of subsets $A$ of $E$ that can be well-ordered and satisfying 1).

Let $A,B\in\mathcal{F}'$ and $C$ the set of $x\in A\cap B$ such that $\se{x}^A=\se{x}^{B}$ and the induced preorderings on that segment coincide.

$C$ is a segment of $A$:  Let $x\in C$ and $y\in\se{x}^A$. Hence $y\in A\cap B$ and $\se{y}^A=\se{y}^B\subset\se{x}^A$. By symmetry, $C$ is also a segment of $B$. Clearly, the induced preordering on $C$ by those of $A$ and $B$ coincide. Suppose that $C\neq A$ and $C\neq B$. Let $x$ be the least element of $A-C$ in $A$ et $x'$ be the least element of $B-C$ in $B$. We have $V=\se{x}^A$: there is $y\in A$ such that $C=\se{y}^A$. We have also $x\leq_Ay$ since $y\in A-C$. Hence $y=x$. By symmetry, $C=\se{x'}^B$. Hence $C\in\mathcal{S}$ and $\varphi(C)=x=x'\in C$. We obtain a contradiction. Finally, $C=A$ or $C=B$, i.e. $A\subset B$ or $B\subset A$.

By Lemma \ref{réunion_Ens_BO}, $M=\bigcup_{A\in\mathcal{F}'}A$ is a well-ordered set satisfying 1). We have also 2), since if we suppose $M\in\mathcal{S}$, then $\omega=\varphi(M)\notin M$, and the well-ordered set $\overline{M}=M\cup\{\omega\}$ satisfy 1) (since $\se{\omega}=M\in\mathcal{S}$), contradiction.

For the unicity, if $M'$ another well-ordered subset of $E$ satisfying 1) and 2), then $M'$ is a segment of $M$, hence $M'=M$.
\end{proof}

\begin{remark}
If $M\neq\varnothing$ and $x$ is the least element of $M$, therefore $\se{x}=\varnothing\in\mathcal{S}$.
\end{remark}

\section{Main results}

\begin{theorem}[Kanamori \cite{Kanamori:1997}]
Let $E$ be a set, $\psi:\mathcal{P}(E)\to E$ be a map. Therefore there is a unique subset $M$ of $E$ that can be well-ordered satisfying
\begin{enumerate}[1)]
  \item for all $x\in M$: $\psi(\se{x})=x$;
  \item $\psi(M)\in M$.
\end{enumerate}
\end{theorem}

\begin{proof}
It suffices to apply Lemma \ref{lem_ZZC} to $\mathcal{S}=\{X\in \mathcal{P}(E)\mid \psi(X)\notin X\}$ and $\varphi=\psi|_\mathcal{S}$. 
\end{proof}

\begin{corollary}
Let $E$ be a set, every map $\psi:\mathcal{P}(E)\to E$ is not injective. 
\end{corollary}

\begin{proof}
By the last theorem, $\psi(\se{\psi(M)})=\psi(M)$ and $\se{\psi(M)}\neq M$. (This proof is due to Kanamori \cite{Kanamori:1997}).
\end{proof}

Let $X$ be a non-empty set. A map $c$ defined on $X-\{\varnothing\}$ is called \emph{choice function} for $X$ if $c(x)\in x$ for all $x\in X-\{\varnothing\}$.

\begin{theorem}[Zermelo \cite{Zermelo:1904}]
If the set $\mathcal{P}(E)$ has a choice function, $E$ can be endowed with a well-ordering.
\end{theorem}

\begin{proof}
It suffices to apply Lemma \ref{lem_ZZC} to $\mathcal{S}=\mathcal{P}(E)-\{E\}$ and $\varphi:\mathcal{S}\to E$ defined by $\varphi(X)=c(E-X)$.
\end{proof}

\begin{theorem}\label{Bourbaki-Moroianu}
Let $E$ be a non-empty preordered set, such that at least one of the following holds:
\begin{enumerate}[1)]
  \item The set $\mathcal{P}(E)$ of all subsets of $E$ has a choice function, and every well-ordered subset of $E$ has an upper bound in $E$.
  \item every well-ordered subset of $E$ has a least upper bound in $E$ (Bourbaki).
\end{enumerate}
If $f:E\to E$ is a map such that $f(x)\geq x$ for all $x\in E$, therefore it has a fixed point.
\end{theorem}

\begin{proof}
We prove the cas 2, and the case 1 can be shown analogously. Let $a\in E$ and $\mathcal{S}$ be the set consisting of the empty-set and the subsets $X$ of $E$ such that $a\in X$, and $X$ has a a least upper bound $m$ in $E$ satisfying $m\not\in X\text{ or }f(m)>m$. Let $\varphi:\mathcal{S}\to E$ be the map defined by $\varphi(\varnothing)=a$ et $$\varphi(X)=\begin{cases}
m & \text{if}\quad m=\sup_EX\notin X\\
f(m) & \text{if}\quad m=\sup_EX\in X.
\end{cases}$$
If $m=\sup_EX\in X$, therefore $f(m)>m$, then $f(m)\not\in X$. By Lemma \ref{lem_ZZC}, there is a subset $M$ of $E$ endowed with a well-ordering such that 1), 2) hold. The canonical injection $j:(M,\leq)\to (E,\leq)$ is strictly increasing: If $x\in\se{y}^M$, then $y=\varphi(\se{y}^M)$ is an upper bound of $\se{y}^M$ in $E$, hence $x<y$. We have $j:M\to M$ is a preordering isomorphism, i.e. the well-ordering on $M$ is the induced preordering on $M$ by that of $E$. We have also $M\neq\varnothing$ since $M\notin\mathcal{S}$, and $a\in M$, since $a\notin M$ imply there exists $x\in M-\{a\}$, which imply $a\in\se{x}^M=\{y\in M\mid y<x\}$ ($\varphi(\se{x}^M)=x\neq a=\varphi(\varnothing)$). Finally, $M$ has a least upper bound $b$ such that $b\in M$ and $f(b)=b$.
\end{proof}

One of the interesting consequences of the Bourbaki's theorem is the following:

\begin{corollary}[Kuratowski's fixed point theorem \cite{Kuratowski:1922}]
Let $X$ be a non-empty set, and $E\subset\mathcal{P}(X)$ such that $A\subset E\Rightarrow \bigcup A\in E$. Every map $f:E\to E$, satisfying $x\subset f(x)$ for all $x\in E$, has a fixed point.
\end{corollary}

\begin{theorem}\label{CKZZ}
In the Zermelo-Fraenkel system, the following are equivalent:
\begin{enumerate}[1)]
  \item The axiom of choice: there is a choice function for $\mathcal{P}(E)$ for all non-empty set $E$ (Zermelo).
  \item If $E$ is a non-empty preordered set such that every well-ordered subset of $E$ has an upper bound in $E$, therefore $E$ a maximal element (H. Kneser).
   \item If $E$ is an inductive preordered set, therefore $E$ a un maximal element (Zorn).
  \item Every set can be endowed with a well-ordering (Zermelo).
\end{enumerate}
\end{theorem}

\begin{proof}
$1)\Rightarrow 2)$. By the axiom of choice, there is a map $f:E\to E$ such that $f(x)=x$ if $x$ is a maximal element, and $f(x)>x$ if not. We conclude by application the case 1 of Theorem~\ref{Bourbaki-Moroianu}.

$2)\Rightarrow 3)$ is trivial.

$3)\Rightarrow 4)$. Let $E$ be a set and $(\mathcal{F},\preceq)$ the inductive preordered set defined in Lemma \ref{réunion_Ens_BO}. Then it has a maximal element $M$. Necessarily $M=E$ (If $M\neq E$, then $\overline{M}=M\cup\{\omega\}\in\mathcal{F}$ and $M\prec\overline{M}$  where $\omega\in E-M$).

$4)\Rightarrow 1)$. Let $E$ be a non-empty set. $E$ can be well-ordered. It suffices to consider the map $$c:\mathcal{P}(E)-\{\varnothing\}\to E,\quad A\mapsto \text{the least element of } A.$$
\end{proof}

\end{document}